\documentclass[reqno,11pt]{amsart}

\usepackage{a4wide,color,eucal,enumerate,mathrsfs}
\usepackage[normalem]{ulem}
\usepackage{amsmath,amssymb,epsfig,amsthm} 
\usepackage[latin1]{inputenc}
\usepackage{psfrag}
\usepackage{enumitem}
\numberwithin{equation}{section}

\newtheorem{theorem}{Theorem}[section]

\newtheorem{corollary}[theorem]{Corollary}
\newtheorem{lemma}[theorem]{Lemma}

\newtheorem{definition}[theorem]{Definition}

\theoremstyle{remark}
\newtheorem{remark}[theorem]{Remark}

\DeclareMathOperator{\diam}{diam}
\DeclareMathOperator{\dist}{dist}

\newcommand{\N}{\mathbb{N}}

\newcommand{\R}{\mathbb{R}}

\newcommand{\sfd}{{\sf d}}
\renewcommand{\d}{{\mathrm d}}

\renewcommand{\d}{{\mathrm d}}

\def\dist{{\mathop\mathrm{\,dist\,}}}

\def\ls{\lesssim}
\def\gs{\gtrsim}

\def\bint{{\ifinner\rlap{\bf\kern.35em--}
\int\else\rlap{\bf\kern.45em--}\int\fi}\ignorespaces}

\def\bbint{{\ifinner\rlap{\bf\kern.35em--}
\hspace{0.078cm}\int\else\rlap{\bf\kern.45em--}\int\fi}\ignorespaces}

\def\diam{{\mathop\mathrm{\,diam\,}}}

\def\bint{{\ifinner\rlap{\bf\kern.35em--}
\int\else\rlap{\bf\kern.45em--}\int\fi}\ignorespaces}

\begin{document}

\title[Density problem for Sobolev spaces in metric measure spaces]
{Density problem for Sobolev spaces on Gehring Hayman domains with the ball separation condition in metric measure spaces}

\author{Jesse Koivu}

\address{University of Jyvaskyla \\
         Department of Mathematics and Statistics \\
         P.O. Box 35 (MaD) \\
         FI-40014 University of Jyvaskyla \\
         Finland}
\email{jesse.j.j.koivu@jyu.fi}

\thanks{The author acknowledges the support from the V\"ais\"al\"a foundation}
\date{\today}



\begin{abstract}
We prove that for a domain $\Omega$ in a PI space $X$ such that $\Omega$ satisfies the Gehring Hayman condition and the ball separation condition, the Newtonian Sobolev space $N^{1,\infty}(\Omega)$ is dense in the space $N^{1,p}(\Omega)$ for $1 < p < \infty$.
\end{abstract}

\maketitle



\section{Introduction}
The concept of density of smooth functions $C^{\infty}(\Omega)$ is a classical topic in the study of Sobolev spaces $W^{1,p}(\Omega)$ on Euclidean domains $\Omega \subset \mathbb{R}^n$. Motivated by the density results of smooth functions, one can ask the following question: Under what assumptions on $p,q$ and $\Omega$ is a Sobolev space $W^{1,q}(\Omega)$ dense in $W^{1,p}(\Omega)$? One answer is in the positive direction by Koskela and Zhang in \cite{KZ2016}, where the authors proved that for a bounded simply connected domain and any $1\leq p < \infty$, it holds that $W^{1,\infty}(\Omega)$ is dense in $W^{1,p}(\Omega)$. They further showed that given the domain is a Jordan domain, then the smooth functions of the plane $C^{\infty}(\mathbb{R}^2)$ are also dense in $W^{1,p}(\Omega)$. Building on the example by Koskela of removable sets (see \cite{K99}) the article of Koskela, Rajala and Zhang \cite{KRZ2017} gave an answer in the negative direction. They found that given an exponent $1 < p < \infty$ there exists a domain $\Omega$ homeomorphic to a unit ball in $\mathbb{R}^3$ via a locally bi-Lipschitz map such that for any $q>p$, the corresponding Sobolev space $W^{1,q}(\Omega)$ is not dense in $W^{1,p}(\Omega)$.
Further, they gave also an answer in the positive direction, showing that whenever $\Omega \subset \mathbb{R}^n$ is a bounded $\delta$-Gromov hyperbolic domain with respect to the quasihyperbolic metric, then for any $1 \leq p < \infty$, it holds that $W^{1,\infty}(\Omega)$ is dense on $W^{1,p}(\Omega)$. They further showed that given the domain is either Jordan (in the case of the plane), or quasiconvex, then also the smooth functions $C^{\infty}(\mathbb{R}^n)$ are dense in $W^{1,p}(\Omega)$.

Considering Gromov hyperbolic domains, the works of Balogh and Buckley \cite{BB2003} and Bonk, Heinonen and Koskela \cite{BHK2001} showed that Gromov hyperbolic domains in Euclidean spaces are precisely those domains that satisfy the ball separation condition, and the Gehring-Hayman condition. In metric spaces it is not known if these are equivalent, but we do know the these two conditions are necessary for Gromov hyperbolicity. The result of Koskela, Rajala and Zhang heavily relies on this equivalence, and in this result we also make heavy use of these two conditions. Here the standing assumption will be that we have a so called $C$-GHS (standing for Gehring-Hayman-Separation with constant C) space that is also a length space. In their paper Balogh and Buckley showed that such a space, that is also locally compact, incomplete and minimally nice, is also a Gromov hyperbolic space.

Taking a step back, let us consider Sobolev spaces on metric measure spaces.
One problem when defining a Sobolev space in such setting, is that there are a multitude of differing definitions, which turn out to be equivalent on Euclidean domains, so each being logical in some sense.
It has been an active topic of research as of late to prove the equivalence of these definitions in the setting of metric measure spaces.
A recent article \cite{AILP2024} by Ambrosio, Ikonen, Lu\v{c}\'ic and Pasqualetto considered different definitions of Sobolev spaces on metric measure spaces, and they proved that for a complete and separable metric space, the most common definitions coincide.
One must note that this equality of definitions only holds for the Sobolev spaces defined on a complete and separable metric space and does not directly imply the same for example the case of non-complete spaces (say a domain).

When studying a Sobolev space, it is a natural question to ask whether results directly generalised from Euclidean spaces hold. Promising results seem to stem in relation to metric spaces that mimic the Euclidean structure. Here we shall consider PI spaces. A PI space is a space that has a doubling measure, and it supports a Poincar\'e inequality with some exponents. Koskela and Haj\l asz in the foundational work \cite{HKSmP} showed that almost all the relevant Sobolev inequalities and embedding theorems have a formulation in the metric setting. However upon stepping outside PI spaces, the situation breaks down as we study spaces of less structure. However even in non-PI spaces one can study Sobolev functions. For the most part in these spaces less is known, or the results turn into counterexamples of varying pathology.

The proof in this article is similar to the Euclidean one, which is dependent on the Whitney decomposition of a domain and smooth partitions of unity.
The concept of smoothness does not come to play when considering Sobolev spaces on metric spaces, and as such it is typical that Lipschitz functions take the role of smooth functions. Furthermore Whitney decompositions with cubes are no longer accessible as metric spaces obey less or none of straight line geometry, so one has to consider similar collections with dyadic balls or the like. Fundamentally these differences do not actually change much.

Given all of this, the main result of this paper is as follows
\begin{theorem}\label{thm:density}
Let $(X,\sfd,\mu)$ be a doubling metric measure space satisfying a local Poincar\'e inequality. Let $\Omega \subset X$ be a C-GHS space and $1 < p < \infty$. Then $N^{1,\infty}(\Omega)$ is dense in $N^{1,p}(\Omega)$.
\end{theorem}

This gives a generalisation of \cite[Theorem 1.3]{KRZ2017} into the metric setting.

\section{Preliminaries}

We will first fix some definitions regarding Sobolev spaces in metric measure spaces and what is meant by a PI space. Throughout the article we assume that $(X,\sfd,\mu)$ is a complete and separable metric measure space, such that $\mu$ is a locally finite Borel measure. 
Furthermore throughout the article whenever we mention 'Sobolev space', we particularly mean the Newtonian Sobolev space $N^{1,p}$.
We begin with the modulus of a curve family as a notion of measuring null sets in curve families. The definition for $1 \leq p < \infty$ follows the exposition of Heinonen, Koskela, Shanmugalingam and Tyson \cite{HKST}. Although the case $p = 1$ is not used here, it is included for completeness. For the case $p = \infty$, our definitions are based on the exposition of Durand-Cartagena and Jaramillo \cite{DCJ}.

\begin{definition}
    Let $1\leq p < \infty$. Let $\Gamma$ be a collection of curves $\gamma:[a,b] \to X$. The p-modulus of the curve family $\Gamma$, denoted by $\textrm{Mod}_p(\Gamma)$, is defined to be the number
    \[
    \inf_\rho \{\lVert \rho\rVert_{L^p}^p : \int_\gamma \rho ds \geq 1 \,\,\textrm{for each rectifiable} \,\, \gamma \in \Gamma\},
    \]
    where the infimum is taken over non-negative Borel measurable functions $\rho$.

    If $p = \infty$, we say the $\infty$-modulus of the curve family $\Gamma$, denoted by $\textrm{Mod}_\infty(\Gamma)$, is defined to be the number
    \[
    \inf_\rho \{\lVert \rho\rVert_{L^\infty}:\int_\gamma \rho \geq 1\,\, \textrm{for all rectifiable} \,\, \gamma\in \Gamma\},
    \]
    where the infimum again is taken over non-negative Borel measurable functions $\rho$.
\end{definition}
Further, we have the definition of an upper gradient and a $p$-weak upper gradient.
\begin{definition}
Let $1 < p \leq \infty$.
    Let $\Omega \subset X$ be an open subset of the metric space $(X,\sfd)$. We say that a function $g \in L^{p}(\Omega)$ is an upper gradient of $u \in L^{p}(\Omega)$, given that for every rectifiable $\gamma:[a,b] \to X$ it holds
    \[
    \lvert u(\gamma(a))-u(\gamma(b))  \rvert \leq \int_\gamma g\,ds. 
    \]
    We say that $g$ is a $p$-weak upper gradient given that the inequality above holds for all rectifiable curves except a null set with respect to $\textrm{Mod}_p$.
\end{definition}

Now with the upper gradient we can define the Sobolev space $N^{1,p}(\Omega)$.

\begin{definition}
Let $1 < p \leq \infty$.
 For an open subset $\Omega\subset X$ we define the space $N^{1,p}(\Omega)$ to consist of functions $u \in L^{p}(\Omega)$ such that there exists a $p$-weak upper gradient $g \in L^{p}(\Omega)$ of $u$. 

    We endow the space $N^{1,p}(\Omega)$ with the norm 
    \[
    \lVert u \rVert_{N^{1,p}(\Omega)} := \lVert u \rVert_{L^p(\Omega)} + \inf \lVert g \rVert_{L^p(\Omega)},
    \]
    where the infimum is taken over all upper gradients $g$ of $u$.
    The minimal (with respect to the $L^p$-norm) weak upper gradient $g$ of $u$ is denoted by $\nabla u$. 
\end{definition}

We define Lipschitz spaces as follows.
\begin{definition}
   We define the space $\textrm{Lip}(X)$ to be the space
    \[
    \textrm{Lip}(X) = \{f:X\to \R : f\,\, \textrm{is}\,\,\textrm{Lipschitz}\}.
    \]
\end{definition}

We look to define PI spaces. The PI assumption consists of two properties, the doubling property and the Poincar\'e inequality. We define these now.

\begin{definition}
    We say that a measure $\mu$ on $X$ is a doubling measure, given that there exists a constant $C > 0$, such that for every ball $B \subset X$ it holds \[
    \mu(2B) \leq C\mu(B).
    \]
    We call a space $(X,\d,\mu)$ a doubling space, given that the measure $\mu$ is doubling.
\end{definition}

\begin{definition}
    We say that $(X,\d,\mu)$ is a $(q,p)$-PI space given that $\mu$ is doubling and $X$ supports a weak local $(q,p)$-Poincar\'e inequality, that is there exists a constant $c_* > 0$ such that for every $u \in L^1_{{\textrm{loc}}}(X)$ and $g$ an upper gradient of $u$ it holds for some $\sigma > 0$
    \[
    \frac{1}{\mu(B)}\Big(\int_B \lVert u(x)-u_B \rVert^q d\mu(x)\Big)^{\frac{1}{q}} \leq c_* \diam(B)\frac{1}{\mu(5\sigma B)}\Big(\int_{5\sigma B} g(x)^pd\mu(x)\Big)^{\frac{1}{p}},
    \]
    here $B$ is any ball, and $u_B$ is the average integral of $u$ over $B$.
\end{definition}





We now look to strengthen the $(1,p)$-Poincar\'e inequality in our setting so that on the right hand side we have also the ball $B$.
This inequality is called the strong local $(1,p)$-Poincar\'e inequality.

\begin{lemma}
    For a geodesic length space $(X,\sfd,\mu)$ that is a $(1,p)$-PI-space it holds
    \[
    \Big(\int_B \lVert u(x)-u_B \rVert d\mu(x)\Big) \leq c_* \diam(B)\Big(\int_{B} g(x)^pd\mu(x)\Big)^{\frac{1}{p}}
    \]
\end{lemma}
    \begin{proof}
        See \cite{HKSmP} Corollary 9.5 and Theorem 9.7.

        Since $X$ is a geodesic length space, it follows that balls $B \subset X$ are John domains.
    \end{proof}

Throughout the rest of the article we assume that $(X,\d,\mu)$ is a $(1,p)$-PI space.

\begin{definition}
Let $\Omega \subset X$ be a domain that is not the whole space. The quasihyperbolic distance between $x,y \in \Omega$ is defined as
\[
\dist_{qh}(x,y) = \inf_\gamma \int_\gamma \dist(z,\partial\Omega)^{-1}dz,
\]
where the infimum is taken over all rectifiable curves $\gamma$ connecting $x,y$. A curve attaining this infimum is called a quasihyperbolic geodesic connecting $x,y$.

A domain $\Omega$ is called $\delta$-Gromov hyperbolic with respect to the quasihyperbolic metric given that for all $x,y,z \in \Omega$ and any corresponding quasihyperbolic geodesics $\gamma_{x,y},\gamma_{z,y},\gamma_{x,z}$ it holds
\[
\dist_{qh}(w,\gamma_{z,y}\cup\gamma_{x,z}) \leq \delta,
\]
for any $w \in \gamma_{x,y}$.
\end{definition}

\begin{remark}
    Given that $X$ is locally compact, there exists quasihyperbolic geodesics between any points $x,y$. This follows from compactness by taking a sequence of curves converging in length to the infimum.
\end{remark}

\begin{definition}
We define the internal metric on an open subset $\Omega$ of $X$, by setting
\[
\sfd_\Omega(x,y) = \inf_{\gamma\subset \Omega} \ell(\gamma).
\]
\end{definition}

For balls defined with respect to the internal distance of a domain $\Omega$ we denote $B_\Omega$.

\begin{definition}
Let $c>0$. We define for an internal metric ball $B$ the relatively closed (scaled) inner metric ball
\[
(c)_\Omega B = \{y\in \Omega: \dist_\Omega(y,x)\leq c\diam(B) \},
\]
where $x$ is the center of the ball $B$.
\end{definition}

\section{Ball separation and the Gehring-Hayman condition}

\subsection{Gromov hyperbolic domains in PI spaces}
We begin by introducing the so called (C-Gehring-Hayman-Separation) C-GHS spaces.

\begin{definition}
We say that a bounded domain $\Omega \subset X$ of a metric space $(X,\d)$ is C-GHS domain (or space) if there exists a constant $C \geq 1$ such that the following conditions hold
\begin{itemize}
\item For any $x,y \in \Omega$ and any quasihyperbolic geodesic $\Gamma$ joining x and y, and every $z \in \Gamma$, the internal metric ball
\[
B = B_{\Omega}(z,C \dist(z, \partial\Omega)),
\]
satisfies $\gamma \,\cap B \neq \emptyset$ for any curve $\gamma \subset \Omega$ that joins $x$ and $y$. We call this property \textit{the C-ball-separation condition}.
\item For any $x,y \in \Omega$, the metric length of each quasihyperbolic geodesic connecting $x$ and $y$ is at most $C\dist_\Omega(x,y)$. We call this condition \textit{the C-Gehring-Hayman condition}.
\item The space $(\Omega, \sfd_{\Omega})$ is a locally compact length space.
\end{itemize}
\end{definition}
Notice that the third condition is fulfilled by a locally compact PI space.
The first two conditions are known to be equivalent with the Gromov hyperbolicity of a domain in Euclidean spaces, see \cite{BB2003} and \cite{BHK2001}.

\section{Whitney-type covering}
\begin{lemma}\label{lem:31}
    Let $\Omega \subset X$ be a domain that is not the whole space.
    There exists a Whitney-type covering for $\Omega$. More specifically, there exists a collection $\{B_i = B_{\Omega}(x_i,r_i) \}_{i\in \N}$ of countably many internal balls in $\Omega$, such that the following conditions hold.
    \begin{enumerate}
        \item $\Omega = \bigcup_{i\in \N} \frac{7}{8}B_i$.
        \item $\sum_{i=1}^{\infty} \chi_{3B_i}(x)\leq N$, for some fixed $N$ and all $x \in X$.
        \item $\frac{1}{2}\diam(B_i) \leq \diam(B_j) \leq 2\diam(B_i)$, whenever $\dist_{\Omega}(B_i,B_j) = 0$
        \item $\diam(B_i) \leq \dist(B_i, \partial\Omega) \leq 4\diam(B_i)$.
    \end{enumerate}
\end{lemma}

\begin{proof}
Define the dyadic annular sections of $\Omega$
\[
A_k = \{x \in \Omega: \dist(x,\partial\Omega) \in [2^{-k},2^{-k+1})\},
\]
and for $k \in \N$ and choose from each $A_k$ a maximal $2^{-k-3}$-separated net $\{x_{k,j}\}_j \subset A_k$. Denote $r(k)=2^{-k-2}$ and let $\{B_k\}_{k\in \N}$, where $B_k = B_{\Omega}(x_{k,j},r(k))$.

Property $(1)$ holds since $\{x_{k,j}\}_j$ is a maximal separated net. Indeed if there is a point $x \in \Omega$ that is not in any of the slightly shrunk balls, it is also $2^{-k-3}$-separated in relation to the other points.

For property $(2)$ we notice that for a ball $B_k = B_{\Omega}(x_{k,j},r(k))$ with $x_{k,j} \in A_k$, it holds $3B_k \cap A_{k+4} = \emptyset$. At the boundary of the annular section an enlargened ball will only touch the next three sections, but since the sections are half open, it will not reach the boundary of the third one.

Property $(3)$ is clear by definition, since the balls that are at a zero distance are in subsequent annular sections and thus their radii $r_1,r_2$ satisfy $\frac{1}{2}r_2 \leq r_1 \leq 2r_2$.

Property $(4)$ also holds by definition of the covering.
 
\end{proof}

We also note the following lemma on connecting internal balls of the covering through chains of balls.
Connecting two balls via a chain of balls we mean that any two points of the given balls can be connected by a curve that exists in the union of the chain.

We use the notation $G(B,\tilde {B})$ for a chain connecting $B,\tilde {B}$, and the notation $cG(B,\tilde {B})$ for the same chain, but each of the balls in the chain are enlarged by a factor $c$. Note here we are taking $B$ and $\tilde {B}$ to be balls in the chain.

\begin{lemma}\label{lem:wc}
    Assume $B$ and $\tilde {B}$ are balls in the Whitney-type covering given in Lemma \ref{lem:31} of $\Omega$ satisfying
    \[
    \frac{1}{c}\diam(B) \leq \diam(\tilde {B}) \leq c\diam(B),
    \]
    for some $c > 1$ and
    \[
    \dist_{\Omega}(B,\tilde {B}) \leq c\diam(B).
    \]
Assume further that the two balls can be joined by a chain of Whitney-type balls of $\Omega$ of diameter larger or equal to $\frac{1}{c}\diam(B)$. Then we find a chain $G_0(B,\tilde {B})$ of at most $C' = C'(c,C)$ Whitney-type balls of $\Omega$ of diameter comparable to $\diam(B)$, connecting the balls $B$ and $\tilde {B}$. Furthermore, the balls in the chain $G_0$ are ordered. Enlarging the balls in the chain so $G(B,\tilde {B}) = \frac{8}{7}G_0(B,\tilde {B})$, we obtain that the balls in the chain $G$ satisfy
\[
\mu(B_i \cap B_{i+1}) \gs_c \mu(B_i),
\]
where $B = B_1$, $\tilde {B} = B_{N_G}$, $i = 1,2,\ldots, N_G-1$ and $N_G$ is the number of balls in the chain $G$.
\end{lemma}
\begin{proof}
The proof follows the Euclidean case in \cite{KRZ2017} Lemma 2.3.
    The C-Gehring-Hayman condition and the assumption $\dist_{\Omega}(B, \tilde {B}) \leq c\diam(B)$
yields a quasihyperbolic geodesic $\gamma$ connecting $B$ and $\tilde {B}$ such that $\ell(\gamma) \ls \diam(B)$. Now the diameters of the Whitney-type balls intersecting $\gamma$ are uniformly bounded from above by a
multiple of $\diam(B)$.
Further, for every Whitney-type ball $B_x$ with $B_x \cap \gamma \neq \emptyset$, by the C-ball-separation condition and
the definition of the Whitney-type balls, namely the property of distance of the Whitney-type ball to the boundary, any other curve connecting $B$ and $\tilde {B}$ must intersect $(4C)_\Omega B_x$.
On the other hand, by our assumption, there exists a sequence of balls connecting $B$ and $\tilde {B}$ with
edge lengths not less than $c^{-1}\diam(B)$. It follows that $\diam(B_x) \gs \diam(B)$.
Finally, for all $B_x \cap \gamma \neq \emptyset$, $\diam(B_x) \sim \diam(B)$ with the constant only depending on c and C.
Since $\ell(\gamma) \ls \diam(B)$ the number of Whitney-type balls intersecting $\gamma$ must be bounded by a constant
depending only on c and C.

We have obtained a chain $\tilde{G}$ of balls that covers the geodesic $\gamma$ and the number of balls is at most $C'(c,C)$. We let $G_0$ be a minimal subchain of $\tilde{G}$ connecting $B$ and $\tilde{B}$. The chain $G_0$ has the following property: The first and the last balls in the chain intersect exactly one other ball in the chain, while all the other balls intersect exactly two, this gives an ordering of the balls in the chain. We now enlarge the balls in the collection, that is let $G(B,\tilde {B}) = \frac{8}{7}G_0(B,\tilde {B})$. From the modification of the chain we obtain the final desired property of overlap of subsequent balls in the collection.
\end{proof}

\begin{corollary}
    In the setting of the previous lemma, it holds
    \[
    \dist_{qh}(B_1,B_2) \leq C'(c,C).
    \]
\end{corollary}


\section{Decomposing the C-GHS domain}

We fix a PI space $(X,\sfd,\mu)$ and C-GHS domain $\Omega \subset X$.

Let $\mathcal{C}$ be the collection of all the Whitney-type covering balls inside $\Omega$.
Say $B_0 \in \mathcal{C}$ is a ball of maximal diameter (where the diameter is with respect to the internal distance).
Recall that a set $O \subset X$ can be partitioned into it's path-components, that is sets which are path-connected.
We denote $\Omega_{m,0}$ to be the path-component of 
\[
\bigcup_{\substack{B\in \mathcal{C}\\ \diam(B) \geq 2^{-m}}} B,
\]
with $B_0 \subset \Omega_{m,0}$.
Denote $\mathcal{C}_{m,0}$ the Whitney-type covering balls that are contained in $\Omega_{m,0}$ and let 
\[
\mathcal{D}_{m,0} = \{B\in \mathcal{C}:B \subset \Omega_{m,0} , B\cap \partial \Omega_{m,0} \neq \emptyset \}.
\]
This is the collection of the covering balls of $\Omega_{m,0}$ that meet its boundary.
We notice that there is some finite number of the covering balls, we relabel the balls in $\mathcal{C}$ so the balls in $\mathcal{D}_{m,0}$ are labeled consecutively $i = 1,\ldots,N$.

Define the union of the covering balls of the boundary as
\[
D_{m,0} = \bigcup_{B_i\in \mathcal{D}_{m,0}} B_i.
\]

We will consider dilated balls from the covering balls of the boundary, that is for $B_i\in \mathcal{D}_{m,0}$ we define (recall $C$ is the constant from the definition of a C-GHS space)
\[
U_j = (5C)_{\Omega}B_j.
\]

We also fix $m$ to be large enough so $U_j \cap B_0 = \emptyset$. Taking a larger $m$ makes the collection $\mathcal{D}_{m,0}$ move closer to the boundary of $\Omega$, and as such it will move away from $B_0$, which is intuitively at the center of $\Omega_{m,0}$.

Now for $B_j$ let $K_j$ (which may be empty) to be the union of all the path components of $\Omega \setminus U_j$ that do not contain $B_0$. This can be loosely interpreted as the collection of points in $\Omega$ whose connection (via paths) to $B_0$ is blocked by the dilated ball $U_j$.

Now we notice that $\Omega$ can be written as a union with respect to the components we defined above
\[
\Omega = \Omega_{m,0} \cup \bigcup_{B_j \in \mathcal{D}_{m,0}} U_j \cup \bigcup_{B_j\in \mathcal{D}_{m,0}} K_j.
\]
This follows from the fact that a geodesic that joins $B_0$ and some point outside $\Omega_{m,0}$ has to pass through the union of the boundary covering balls $D_{m,0}$ and the C-ball-separation condition. The C-ball-separation condition implies we can only go with a path from a point $x$ in a block component $K_j$ to $B_0$ through some $U_k$, from which it follows that $x \in K_k$.
More precisely, take a quasihyperbolic geodesic $\gamma_x$ connecting $B_0$ to a point $x \in \Omega$. We may assume $x \notin \Omega_{m,0}$. We must show that $x \in \bigcup_{B_j \in \mathcal{D}_{m,0}} U_j \cup \bigcup_{B_j\in \mathcal{D}_{m,0}} K_j$. Assume $x \notin U_j$ for all $j$. Now since $\gamma_x \cap D_{m,0} \neq \emptyset$, we find a $B_k$ that the geodesic $\gamma_x$ passes through. Now the $C$-ball-separation gives us that all curves connecting $B_0$ to the point $x$ have to pass through $B_\Omega(z, C\dist(z,\partial \Omega))$, and since $\dist(z,\partial \Omega) \leq 5\diam(B_k)$ for $B_k$ with $z \in B_k$, we obtain that $U_k$ has to block the curves passing from $B_0$ to $x$. Thus $x \in K_k$.

Suppose there exists a $B_k \in \mathcal{D}_{m,0}$ such that $U_j \cap U_k = \emptyset$ and $K_j \cap U_k \neq \emptyset$, then by the path-connectedness of $U_k \subset \Omega \setminus U_j$ and the definition of $K_j$ we have $U_k \subset K_j$.

We now refine the decomposition obtained above one by one according to the blocks $K_j$. The goal is to cut away the blocks from the center $\Omega_{m,0}$ one by one.

To this end define \[
\mathcal{C}_{m,1} = \{B \in \mathcal{C}_{m,0}: B \not\subset (K_1\setminus(25C)_{\Omega}B_1)\},
\]
and let 
\[\Omega_{m,1} = \bigcup_{B \in \mathcal{C}_{m,1}} B \subset \Omega_{m,0},
\]
and
\[
\mathcal{D}_{m,1} = \{B\in \mathcal{D}_{m,0}: B \subset \Omega_{m,1}, B \cap \partial \Omega_{m,1} \neq \emptyset \}.
\]

Now we claim that the decomposition
\[
\Omega = \Omega_{m,1} \cup \bigcup_{B_j \in \mathcal{D}_{m,1}} U_j \cup \bigcup_{B_j\in \mathcal{D}_{m,1}} K_j,
\]
still holds.
Consider first $y \in K_k$ with $B_k \notin \mathcal{D}_{m,1}$. Now $B_k \subset K_1\setminus (25C)_\Omega B_1$. Recall that $2^{-m} \leq \diam(B_k) < 2^{-m+2}$ for any $Q_k \in \mathcal{D}_{m,0}$. Now $U_k \cap U_1 = \emptyset$ and thus $U_k \subset K_1$.
If $y \in U_k$ with $B_k \notin \mathcal{D}_{m,1}$, then $B_k \subset K_1\setminus (25C)_\Omega B_1$.
Finally if $y \in \Omega_{m,0} \setminus \Omega_{m,1}$ then $y$ is contained in a ball in $\mathcal{C}_{m,0}$, but not $\mathcal{C}_{m,1}$, so evidently $y \in K_1\setminus (25C)_\Omega B_1$.
In each of the three cases we thus have $y \in K_1$.

We continue recursively, let $j \in \{ 2,3,\ldots,N\}$. 

If $B_j \notin \mathcal{D}_{m,j-1}$, we let $\Omega_{m,j} = \Omega_{m,j-1}$, $\mathcal{C}_{m,j} = \mathcal{C}_{m,j-1}$
and $\mathcal{D}_{m,j} = \mathcal{D}_{m,j-1}$.

Otherwise we let, as above, 
\[
\mathcal{C}_{m,j} = \{B\in \mathcal{C}_{m,j-1} : B \not\subset (K_j \setminus (25C)_\Omega B_j) \},
\]
\[
\Omega_{m,j} = \bigcup_{B \in \mathcal{C}_{m,j}} B \subset \Omega_{m,j-1} ,
\]
and
\[
\mathcal{D}_{m,j} = \{B \in \mathcal{D}_{m,j-1}: B \subset \Omega_{m,j}, B \cap \partial\Omega_{m,j} \neq \emptyset \}.
\]

Finally let $\Omega_m = \Omega_{m,N}$ to obtain the figurative 'center' of the domain.

Now any Whitney-type covering ball in $\mathcal{C}_{m,N}$ that intersects $\partial\Omega_m$ is contained in $(60C)_\Omega B_j$ for some $j$ and has edge length comparable to $2^{-m}$ with constant $c(C)$. Further there is a constant $M(C,d)$ with $\Omega_m \subset\subset \Omega_{m'}$ whenever $m' \geq m+M(C)$. Notice that $\Omega = \bigcup_m \Omega_{m,0}$, this and the above we find
$\Omega = \bigcup_m \Omega_m$. Denote $\mathcal{D}_{m} = \mathcal{D}_{m,N}$.

We have obtained

\begin{lemma}\label{lem:wprop}
Let $\Omega \subset X$ be a C-GHS domain (with respect to the quasihyperbolic metric), the collection $\mathcal{C}$ the Whitney-type covering balls of $\Omega$ as defined in Lemma \ref{lem:31} and let $B_0$ be one of the largest scale ones.

There exists a sequence of sets $\Omega_m \subset\subset \Omega$ such that, by setting
\[
\mathcal{B}_m = \{B\in \mathcal{C}:B\subset\Omega_m, B\cap \partial\Omega_m \neq \emptyset \}
\]
and by letting $U_j = (5C)_\Omega B_j$ for all $B_j \in \mathcal{B}_m$ and $K_j$ as the union of all path-components of $\Omega \setminus U_j$ that do not contain $B_0$, the following conditions hold:

\begin{enumerate}[label=(\alph*)]
    \item $\Omega_m$ consists of finitely many Whitney-type balls and any two of them can be joined by a chain of Whitney-type balls in $\Omega$ of radius $r \geq 2^{-m}$. Further $B_0 \subset \Omega_m$ and there exists a constant $M = M(C)$ such that $\Omega = \bigcup_m \Omega_m$ and $\Omega_m \subset\subset \Omega_{m'}$ for any $m' \geq m+M$.
    \item For every $B \in \mathcal{B}_m$ it holds $2^{-m} \leq \diam(B) \ls 2^{-m}$.
    \item There exists a subcollection $\mathcal{D}_m$ of $\mathcal{B}_m$ such that for each $B_k \in \mathcal{B}_m$ and $B_j\in\mathcal{D}_m$ with $B_k \cap K_j \neq \emptyset$ it holds $B_k \subset (60C)_\Omega B_j$. Furthermore $\{(60C)_\Omega B_j\}_{B_j\in \mathcal{D}_m}$ covers all the boundary balls of $\Omega_m$.
    \item It holds 
    \[
    \Omega = \Omega_m \cup \bigcup_{B_j \in \mathcal{D}_m}U_j \cup \bigcup_{B_j\in \mathcal{D}_m} K_j.
    \]
\end{enumerate}
\end{lemma}

Later on we will apply the Poincar\'e inequality on the chains of Whitney-type balls. Here we will formulate this as a lemma. First denote the average of a function $u$ over a Whitney-type ball $B_j \in \mathcal{B}_m$ as
\[
a_j(u) = \frac{1}{\mu(B_j)}\int_{B_j} u(x)\,d\mu(x).
\]
We suppress the notation into $a_j$ whenever the function $u$ is clear from the context.
We use Lemma $\ref{lem:wc}$ on the balls $B_j,B_k$ with diameters comparable to $2^{-m}$.
First let $B_j = B$ and $B_k = \tilde{B}$. We obtain the chain $G(B_j,B_k) = G(B,\tilde{B}) = \{\tilde{B}_i\}_{i=1}^{N}$, for which the overlap estimate $\mu(\tilde{B}_i \cap \tilde{B}_{i+1}) \gs \mu(\tilde{B}_i)$ holds.

\begin{lemma}[Poincar\'e on chains]\label{lem:poinchain}
Let $u \in W^{1,p}(\Omega)$ and
let $B_j, B_k \in \mathcal{B}_m$ be Whitney-type balls. Let $G(B_j,B_k) = G(B,\tilde{B})$ be the chain of Whitney-type balls given in Lemma $\ref{lem:wc}$ with diameter comparable to $2^{-m}$. It holds
\[
\int_{{B}_j} \lvert a_j-a_k\rvert^p d\mu(x) \ls 2^{-mp}\mu({B}_j)\sum_{i=1}^{N-1}\frac{1}{\mu(\tilde{B}_i)}\int_{B_i} \lvert \nabla u\rvert^p d\mu.
\]
\end{lemma}
\begin{proof}

The left hand side can be estimated from above by 
    \begin{equation}
    \begin{split}
    \int_{{B}_j} \lvert a_j-a_k\rvert^p d\mu(x) & \ls \sum_{i=1}^{N-1} \int_{{B}_j}\lvert a_i - a_{i+1}\rvert^p d\mu = \sum_{i=1}^{N-1} \mu({B}_j)\lvert a_i - a_{i+1}\rvert^p.
    \end{split}
    \end{equation}
We estimate the term in the sum

\begin{equation}
    \begin{split}
     \lvert a_i -a_{i+1}\rvert^p =
     & \Big(\frac{1}{\mu(\tilde{B}_i \cap \tilde{B}_{i+1})}\int_{\tilde{B}_i \cap \tilde{B}_{i+1}} \lvert a_i-a_{i+1}\rvert d\mu\Big)^p \\
     \ls & \Big(\frac{1}{\mu(\tilde{B}_i)}\int_{\tilde{B_i}} \lvert u-a_i \rvert d\mu \Big)^p + \Big(\frac{1}{\mu(\tilde{B}_{i+1})}\int_{\tilde{B}_{i+1}} \lvert u-a_{i+1} \rvert d\mu\Big)^p.
     \end{split}
\end{equation}

Using the $(1,p)$-Poincar\'e inequality on these terms yields
\[
 \lvert a_i-a_{i+1}\rvert^p \ls 2^{-mp} \Big(\frac{1}{\mu(\tilde{B}_i)}\int_{\tilde{B}_i}\lvert\nabla u\rvert^p d\mu + \frac{1}{\mu(\tilde{B}_{i+1})}\int_{\tilde{B}_{i+1}}\lvert\nabla u\rvert^p d\mu\Big).
\]

\end{proof}

\begin{corollary}\label{cor:poin}
    It holds
    \[
2^{mp}\int_{B_j} \lvert a_j-a_k\rvert^p d\mu(x) \ls \int_{G(B_j,B_k)} \lvert \nabla u\rvert^p d\mu.
\]
\end{corollary}
\begin{proof}
    The measures of the balls are all comparable to $2^{-m}$ and there are uniformly finitely many balls in the chain $G(B_j,B_k)$ for any pair of balls $B_j,B_k$.
\end{proof}

We next decompose $\Omega \setminus \Omega_m$, that is the part of $\Omega$ near the boundary.

First we let 
\[
E_m = \Bigg[ \bigcup_{B_j \in \mathcal{D}_m} (70C)_{\Omega} B_j \Bigg] \setminus \Omega_m,
\]
and further define the rest of the set as
\[
F_m = \Omega \setminus (\Omega_m \cup E_m).
\]

By definition it holds $\dist_{\Omega}(F_m,\Omega_m) \geq 2^{-m}$,
and 
\[
F_m \subset \bigcup_{B_j \in \mathcal{D}_m} K_j.
\]

We set the following notation for sufficiently thin neighborhoods (with respect to the internal metric) of sets
\[
\mathcal{N}_{m,\Omega}(A) = \{x \in \Omega: \dist_{\Omega}(x,A) \leq 2^{-5-m}\}.
\]

We look to decompose the sets $E_m,F_m$ into disjoint sets and then show that the neighborhoods of these sets do not overlap much.

We now decompose the set $E_m$.

To this end notice that \[
E_m \subset \bigcup_{B_j \in \mathcal{D}_m}(71C)_{\Omega}B_j.
\]
Let $V_j = (71C)_{\Omega}B_j$ for each $B_j \in \mathcal{D}_m$ and assume the sets in $\mathcal{D}_m$ are ordered $\{B_1, \ldots, B_{\tilde{N}}\}$. 
For a fixed $V_j$ it holds
\[
\#\{1\leq k\leq \tilde{N}: \dist_{\Omega}(V_j,V_k) \leq 2^{-m-4} \} \leq C'(C),
\]
Which follows from the fact that $\dist_{\Omega}(B_j,B_k) \ls 2^{-m}$ and $\diam(B_j) \sim 2^{-m}$ and both the constants are independent of the index $j$.

Define now the disjoint sets decomposing $E_m$ as follows: Let \[S_1 = V_1 \cap E_m, \] and then set \[S_j = \Big(V_j\setminus \bigcup_{i=1}^{j-1}V_i \Big) \cap E_m.\]

What we note is that the sets $S_j$ still cover $E_m$ 
and have the same overlap bound, that is
\[
\#\{1\leq k\leq \tilde{N}: \dist_{\Omega}(S_j,S_k) \leq 2^{-m-4} \} \leq C'(C).
\]

We also note that \[
\#\{k\in \N: B_j \in \mathcal{B}_m, B_j \cap S_k \neq \emptyset\} \leq C'(C),
\]
and $S_j\cap S_k = \emptyset$ and $\diam_\Omega (S_j) \ls 2^{-m}$.

Let us proceed with the decomposition of $F_m$. To this end we first define $T_j^{'} = K_j \cap F_m$.
It might be the case that $T_j^{'}$ is empty. Further it holds that $F_m = \bigcup_j T_j^{'}$.

Fix now $T_j^{'}$ and suppose $\dist_{\Omega}(T_j^{'}, S_k) \leq 2^{-m-4}$. It then holds $\dist_{\Omega}(V_j, V_k) \leq 2^{-m-4}$.
We check this. Assume that $\dist_{\Omega}(V_j, V_k) > 2^{-m-4}$. Now also $\dist_{\Omega}(U_j, V_k) > 2^{-m-4}$ and by the fact that $V_k$ is path-connected we can connect a point $x \in T_j^{'}\cap V_k$ to $B_k \in \mathcal{D}_m$ by a path in $V_k \subset \Omega \setminus U_j$. If $B_k$ can be connected to $B_0$ via a path in $\Omega \setminus U_j$, then also $x$ can be connected to $B_0$ via a path in $\Omega \setminus U_j$, which is a contradiction with the definition of the block $K_j$ which contains $T_j^{'}$. On the other hand if $B_k$ can not be connected to $B_0$ via a path in $\Omega \setminus U_j$, then $B_k \subset K_j$, then Lemma $\ref{lem:wprop}$ gives \[
\dist_{\Omega}(B_k,B_j) \leq 60C \diam(B_j).
\]

From these we obtain the overlap bound, it holds
\[
\#\{1\leq k\leq \tilde{N}: \dist_{\Omega}(T_j^{'},S_k) \leq 2^{-m-4}\} \leq C'(C).
\]
Note that if $y \in T_j^{'}\cap T_k^{'}$, then the path component of $T_j^{'}$ that contains $y$ is a subset of $T_j^{'}\cap T_k^{'}$.

We now define $T_1 = T_1^{'}$ and then set 
\[
T_j = T_j^{'} \setminus \bigcup_{i=1}^{j-1} T_i^{'}.
\]

Assume $T_j$ is fixed. Then it holds
\[
\#\{1\leq k\leq \tilde{N}: \dist_{\Omega}(T_j,S_k) \leq 2^{-m-4}\} \leq C'(C),
\]
and further for a fixed $S_j$ it holds
\[
\#\{1\leq k\leq \tilde{N}: \dist_{\Omega}(S_j,T_k) \leq 2^{-m-4}\} \leq C'(C).
\]

Finally we conclude that whenever $\dist_{\Omega}(B_j, S_k) < 2^{-m-4}$ with $B_j \in \mathcal{B}_m$, $\dist_{\Omega}(S_j, S_k) < 2^{-m-4}$ and $\dist_{\Omega}(T_j, S_k) < 2^{-m-4}$, then the corresponding $B_j, B_k$ satisfy $\dist_{\Omega}(B_j, B_k) \ls 2^{-m}$.

Let us gather from the above the overlap estimates for neighborhoods of the decomposition of $\Omega$.

\begin{lemma}\label{lem:overlap}
    For the decomposition of $\Omega$ given above and for each $B_j \in \mathcal{D}_m$ we have the following estimates
\begin{enumerate} 
    \item For each $B \in \mathcal{B}_m$ it holds 
\[ \#\{ 1\leq k \leq \tilde{N}: B \cap \mathcal{N}_{m,\Omega}(S_k) \neq \emptyset\} \leq C'(C),
\]
    \item $
\#\{ 1\leq k \leq \tilde{N}: \mathcal{N}_{m,\Omega}(S_j) \cap \mathcal{N}_{m,\Omega}(S_k) \neq \emptyset\} \leq C'(C),
$
    \item $
\#\{ 1\leq k \leq \tilde{N}: \mathcal{N}_{m,\Omega}(S_j) \cap \mathcal{N}_{m,\Omega}(T_k) \neq \emptyset\} \leq C'(C),
$
    \item $
\#\{ 1\leq k \leq \tilde{N}: \mathcal{N}_{m,\Omega}(T_j) \cap \mathcal{N}_{m,\Omega}(S_k) \neq \emptyset\} \leq C'(C),
$
    \item $
\#\{ 1\leq k \leq \tilde{N}: \mathcal{N}_{m,\Omega}(T_j) \cap \mathcal{N}_{m,\Omega}(T_k) \neq \emptyset\} \leq C'(C).
$
\end{enumerate}
\end{lemma}
\begin{proof}
\begin{enumerate}
\item First recall that
\[
\#\{k\in \N: B \in \mathcal{B}_m, B \cap S_k \neq \emptyset\} \leq C'(C),
\]
from which it also follows 
\[
\#\{ 1\leq k \leq \tilde{N}: B \cap \mathcal{N}_{m,\Omega}(S_k) \neq \emptyset\} \leq C'(C).
\]

\item Now fix $B_j \in \mathcal{D}_m$, this stays fixed for the remaining items. We first recall that it holds
\[
\#\{1\leq k \leq \tilde{N}: \dist_{\Omega}(V_j,V_k) \leq 2^{-m-4} \} \leq C'(C),
\]

and

\[
\#\{1\leq k \leq \tilde{N}: \dist_{\Omega}(S_j,S_k) \leq 2^{-m-4} \} \leq C'(C).
\]

From this and the fact that for each $B \in \mathcal{B}_m$ it holds that $2^{-m} \leq \diam{B} \ls 2^{-m}$, it follows

\[
\#\{ 1\leq k \leq \tilde{N}: \mathcal{N}_{m,\Omega}(S_j) \cap \mathcal{N}_{m,\Omega}(S_k) \neq \emptyset\} \leq C'(C).
\]

\item Recall that it holds
\[
\#\{1\leq k \leq \tilde{N}: \dist_{\Omega}(S_j,T_k^{'}) \leq 2^{-m-4} \} \leq C'(C),
\]
and directly from this it holds
\[
\#\{1\leq k \leq \tilde{N}: \dist_{\Omega}(S_j,T_k) \leq 2^{-m-4} \} \leq C'(C).
\]
From this we get the estimate
\[
\#\{ 1\leq k \leq \tilde{N}: \mathcal{N}_{m,\Omega}(S_j) \cap \mathcal{N}_{m,\Omega}(T_k) \neq \emptyset\} \leq C'(C).
\]

\item Similarly as in the previous item we have also
\[
\#\{1\leq k \leq \tilde{N}: \dist_{\Omega}(T_j,S_k) \leq 2^{-m-4} \} \leq C'(C),
\]
and from this
\[
\#\{ 1\leq k \leq \tilde{N}: \mathcal{N}_{m,\Omega}(T_j) \cap \mathcal{N}_{m,\Omega}(S_k) \neq \emptyset\} \leq C'(C).
\]
\item Finally we have also
\[
\#\{1\leq k \leq \tilde{N}: \dist_{\Omega}(T_j^{'},T_k^{'}) \leq 2^{-m-4} \} \leq C'(C),
\]
from which it directly follows that
\[
\#\{1\leq k \leq \tilde{N}: \dist_{\Omega}(T_j,T_k) \leq 2^{-m-4} \} \leq C'(C).
\]

Notice that $\mu(\mathcal{N}_{m,\Omega}(S_j) \setminus S_j) \ls 2^{-m}$ uniformly in $j$. Similarly by the proof of the overlap bound for $S_k$ and $T_j^{'}$ we obtain that there is uniformly finitely many $V_k$ such that $V_k$ and $\mathcal{N}_{m,\Omega}(T_j) \setminus T_j$ intersect. Thus also $\mu(\mathcal{N}_{m,\Omega}(T_j) \setminus T_j) \ls 2^{-m}$. A similar argument as the one for the overlap bound for $S_k$ and $T_j^{'}$ yields the bound
\[
\#\{ 1\leq k \leq \tilde{N}: \mathcal{N}_{m,\Omega}(T_j) \cap \mathcal{N}_{m,\Omega}(T_k) \neq \emptyset\} \leq C'(C).
\]
\end{enumerate}
\end{proof}

We construct a partition of unity in $\Omega$ according to the decomposition given above.

\begin{lemma}\label{lem:p1}
    There exist functions $\psi, \phi_j, \varphi_j$, $1 \leq k \leq \tilde{N}$ such that
    \begin{itemize}
        \item The function $\psi$ is Lipschitz in $\Omega$, and compactly supported in $\Omega_m$, with $0 \leq \psi \leq 1$ and $\lvert  \nabla \psi(x)\rvert \ls 2^{-m}$. Here $\nabla$ stands for the minimal upper gradient.
        \item For each $j$ we have $\phi_j \in W^{1,\infty}(\Omega)$ and the support of $\phi_j$ is relatively closed in $\Omega$ and contained in $\mathcal{N}_{m,\Omega}(S_j)$, with $0 \leq \phi_j \leq 1$ and $\lvert \nabla \phi_j \rvert \ls 2^{-m}$.
        \item For each $j$ we have $\varphi_j \in W^{1,\infty}(\Omega)$ and the support of $\varphi_j$ is relatively closed in $\Omega$ and contained in $\mathcal{N}_{m,\Omega}(T_j)$, with $0 \leq \varphi_j \leq 1$ and $\lvert \nabla \varphi_j \rvert \ls 2^{-m}$.
        \item $\psi(x) + \sum_j \phi_j(x) + \sum_j \varphi_j(x) = 1$ for all $x \in \Omega$.
    \end{itemize}
\end{lemma}
\begin{proof}
    We define cut-offs with respect to the internal distance as follows. Let
    \[
    \phi_j(x) = \max\{1-2^{m+6}\dist_{\Omega}(x,S_j), 0\},
    \]
    and similarly
    \[
    \varphi_j(x) = \max\{1-2^{m+6}\dist_{\Omega}(x,T_j), 0\},
    \]
    and finally define
    \[
    \psi(x) = \min \{2^{m+8} \dist_{\Omega}(x,E_m\cup F_m), 1\}.
    \]

    It holds that 
    \[
    \psi(x)+\sum_{j}\phi_j(x) + \sum_j \varphi_j(x) \geq \chi_{\Omega}(x).
    \]

Now by the previous overlap lemma and the decomposition of $\Omega$ we have 
\[
\chi_{\Omega}(x) \leq \psi(x)+\sum_{j}\phi_j(x) + \sum_j \varphi_j(x) \ls \chi_{\Omega}(x),
\]

and thus by normalizing the terms to sum to one, that is by dividing each of the terms by the whole sum we obtain the desired partition of unity.
\end{proof}

\section{The proof of density}
Here we prove Theorem \ref{thm:density} using the decomposition and partition of unity from the previous sections. We recall the statement of the theorem here.

\begin{theorem}\label{thm:density_rec}
Let $(X,\sfd,\mu)$ be a $(1,p)$-PI space.
Let $\Omega \subset X$ be a C-GHS space and $1 < p < \infty$. Then $N^{1,\infty}(\Omega)$ is dense in $N^{1,p}(\Omega)$.
\end{theorem}
\begin{proof}
    Fix $\varepsilon > 0$ and $v \in N^{1,p}(\Omega)$. By truncating and normalizing we assume that $\lVert v\rVert_{L^{\infty}(\Omega)} = 1$. Let $m \in \N$, to be fixed.
    Recall by Lemma \ref{lem:wprop} that $\Omega = \Omega_m \cup E_m \cup F_m$. We define $D_m'$ to be the union of the Whitney type balls $B \in \mathcal{C}$ for which there is a chain of less or equal to $N_0(C)$ Whitney-type balls joining $B$ to a ball in $\mathcal{B}_m$. Here the constant is given by Lemma $\ref{lem:wc}$ using the fact that the diameters of the $B \in \mathcal{B}_m$ satisfy $2^{-m} \leq \diam{B} \ls 2^{-m}$.
    Now the quasihyperbolic distance from $B$ to $\bigcup_{B\in \mathcal{B}_m} B$ is uniformly bounded given $B \subset D_m'$. Now for any Whitney type ball $B \subset D_m'$ it holds
    \[
    \diam(B) \ls_C 2^{-m}.
    \]
    Further from Lemma \ref{lem:wprop} it holds $\mu(\Omega\setminus \Omega_m) \to 0$ as $m \to \infty$. Thus for $m$ large enough we have
    \begin{equation}\label{eq:123}
    \lVert v\rVert_{N^{1,p}(D_m' \cup E_m \cup F_m)}^{p} \leq \varepsilon \quad {\rm and} \quad \mu(E_m\cup F_m) \leq \varepsilon.
    \end{equation}

    We find a uniform domain $\Omega' \subset \Omega$ containing $\Omega_m$ (see \cite{R2020}). Since uniform domains are extension domains (see \cite{BS2007}) on PI spaces, we can extend the function $v |_{\Omega'}$ to a function $Ev|_{\Omega'} = w \in N^{1,p}(X)$. Now Lipschitz functions are dense in $N^{1,p}(X)$ for a PI space $X$ (see for example \cite{LP2024}), so we find a bounded $u \in \textrm{Lip}(X)$ such that 
    \[\lVert w-u\rVert_{N^{1,p}(X)} < \varepsilon.
    \]  It holds $u \in N^{1,\infty}(X)$ (see \cite[Corollary 5.24]{DCJ}). Now $u \in \textrm{Lip}(\Omega_m)$ and $u \in N^{1,\infty}(\Omega_m)$. Notice that \[
    \lVert u-v\rVert_{N^{1,p}(\Omega_m)} < \varepsilon.
    \]

    We now define $u_m$ on $\Omega$ by setting
    \[
    u_m(x) = u(x)\psi(x) + \sum_j a_j\phi_j(x) + \sum_j a_j\varphi_j(x),
    \]
    where the functions $\psi, \phi_j, \varphi_j$ are the functions from the partition of unity and $a_j$ is the integral average of $u$ over $B_j \in \mathcal{B}_m$, i.e.
    \[
    a_j = \frac{1}{\mu(B_j)}\int_{B_j} u(x)\,d\mu(x).
    \]

The estimate $\lVert u_m - v\rVert_{N^{1,p}(\Omega_m)} < \varepsilon$ still holds after taking the partition of unity.

Estimating the norm of the difference $u_m - v$ we obtain the estimate
\begin{equation}\label{eq:est}
\begin{split}
\lVert u_m-v\rVert_{N^{1,p}(\Omega)} \leq &\,\, \lVert u_m\rVert_{L^{p}(\Omega\setminus\Omega_m)} + \lVert \nabla u_m \rVert_{L^p(D_m' \cup E_m \cup F_m)} \\ &+ \lVert v\rVert_{N^{1,p}(D_m' \cup E_m \cup F_m)} + \lVert u_m-v\rVert_{N^{1,p}(\Omega_m)}.
\end{split}
\end{equation}

Notice that $u_m \in N^{1,\infty}(\Omega)$, since by boundedness of $\Omega$ there is finitely many $B_j \in \mathcal{B_m}$ and Lemma $\ref{lem:p1}$ gives sufficient estimates on the upper gradients. Furthermore $\lVert u_m\rVert_{L^{\infty}(\Omega)} \leq 1$, by the definition of $u_m$ and again Lemma $\ref{lem:p1}$ and the assumption on $u$. Thus we deduce $\lVert u_m \rVert_{L^p(\Omega \setminus \Omega_m)} \leq \varepsilon$, by the estimate on the $N^{1,p}(D_m^{'}\cup E_m\cup F_m)$ norm of $u$. Now using the estimates on the $N^{1,p}$-norms on the inequality (\ref{eq:est}), the definition of $\mathcal{N}_{m,\Omega}(\cdot)$ and again Lemma $\ref{lem:p1}$ we only need to show that 
\[
\int_{(\Omega \setminus \Omega_m)\cup (\bigcup_{B\in \mathcal{B}_m} B)} \lvert \nabla u_m \rvert^p d\mu \ls \varepsilon.
\]

To this end, write $G(B_j,B_k)$ for the union of Whitney type balls given in Lemma $\ref{lem:wc}$. Recall that $\dist_{\Omega}(\Omega_m, F_m) \geq 2^{-m}$. We will look to control the integrals of the upper gradient of $u_m$ with integrals of the upper gradient of $u$. Recall that by Lemma $\ref{lem:wc}$ there is a uniformly finite amount of balls contained in the chain $G(B_j,B_k)$ whenever $\mathcal{N}_{m,\Omega}(S_k) \cap B_j \neq \emptyset$. Using the Overlap Lemma $\ref{lem:overlap}$, Lemmas $\ref{lem:wc}$, $\ref{lem:p1}$ and Corollary $\ref{cor:poin}$
we compute for $B_j \in \mathcal{B}_m$

\begin{equation}
    \begin{split}
        &\int_{B_j} \lvert \nabla u_m\rvert^pd\mu\ls \int_{B_j} \lvert \nabla (u_m - a_j)\rvert^pd\mu \\
        \ls & \int_{B_j} \lvert \nabla [(u(x)-a_j)\psi(x)]\rvert^pd\mu(x) + \sum_{\substack{S_k \subset E_m \\ B_j \cap \mathcal{N}_{m,\Omega}(S_k) \neq \emptyset}} \int_{B_j} \lvert \nabla [(a_k-a_j)\phi_k(x)]\rvert^pd\mu(x) \\
        \ls & \int_{B_j} \lvert \nabla u\rvert^p + \lvert u(x)-a_j\rvert^p 2^{mp}d\mu(x) + \sum_{\substack{S_k \subset E_m \\ B_j \cap \mathcal{N}_{m,\Omega}(S_k) \neq \emptyset}} \int_{B_j} \lvert a_k-a_j\rvert^p 2^{mp}d\mu(x) \\
        \ls & \int_{B_j} \lvert \nabla u\rvert^pd\mu + \sum_{\substack{S_k \subset E_m \\ B_j \cap \mathcal{N}_{m,\Omega}(S_k) \neq \emptyset}} \int_{G(B_j,B_k)} \lvert \nabla u \rvert^pd\mu.
    \end{split}
\end{equation}

We make a similar estimate for the upper gradient in $S_j$. First recall that $\psi$ is compactly supported in $\Omega_m$. Using now the Overlap Lemma $\ref{lem:overlap}$, Lemmas $\ref{lem:p1}$, $\ref{lem:wc}$, $\ref{lem:poinchain}$ and the fact that $\diam(S_j) \ls 2^{-m}$ we compute for $S_j$

\begin{equation}
    \begin{split}
        &\int_{S_j} \lvert \nabla u_m\rvert^pd\mu\ls \int_{S_j} \lvert \nabla (u_m - a_j)\rvert^pd\mu \\
        \ls & \sum_{\substack{S_k \subset E_m \\ \mathcal{N}_{m,\Omega}(S_j) \cap \mathcal{N}_{m,\Omega}(S_k) \neq \emptyset}} \int_{S_j} \lvert \nabla [(a_k-a_j)\phi_k(x)]\rvert^pd\mu(x) \\
        & + \sum_{\substack{T_k \subset F_m \\ \mathcal{N}_{m,\Omega}(S_j) \cap \mathcal{N}_{m,\Omega}(T_k) \neq \emptyset}} \int_{S_j} \lvert \nabla [(a_k-a_j)\varphi_k(x)]\rvert^pd\mu(x)\\
        \ls & \sum_{\substack{S_k \subset E_m \\ \mathcal{N}_{m,\Omega}(S_j) \cap \mathcal{N}_{m,\Omega}(S_k) \neq \emptyset}} \int_{S_j}\lvert a_k-a_j\rvert^p 2^{mp}d\mu +\sum_{\substack{S_k \subset E_m \\ \mathcal{N}_{m,\Omega}(S_j) \cap \mathcal{N}_{m,\Omega}(T_k) \neq \emptyset}} \int_{S_j}\lvert a_k-a_j\rvert^p 2^{mp}d\mu \\
        \ls & \sum_{\substack{S_k \subset E_m \\ \mathcal{N}_{m,\Omega}(S_j) \cap \mathcal{N}_{m,\Omega}(S_k) \neq \emptyset}} \int_{G(B_j,B_k)} \lvert \nabla u \rvert^p d\mu + \sum_{\substack{S_k \subset E_m \\ \mathcal{N}_{m,\Omega}(S_j) \cap \mathcal{N}_{m,\Omega}(T_k) \neq \emptyset}} \int_{G(B_j,B_k)} \lvert \nabla u \rvert^p d\mu.
    \end{split}
\end{equation}

Analogously for $T_j$ we obtain 
\begin{equation}
\begin{split}
&\int_{T_j} \lvert \nabla u \rvert^pd\mu \\
\ls &\sum_{\substack{S_k \subset E_m \\ \mathcal{N}_{m,\Omega}(T_j) \cap \mathcal{N}_{m,\Omega}(S_k) \neq \emptyset}} \int_{G(B_j,B_k)} \lvert \nabla u \rvert^p d\mu + \sum_{\substack{S_k \subset E_m \\ \mathcal{N}_{m,\Omega}(T_j) \cap \mathcal{N}_{m,\Omega}(T_k) \neq \emptyset}} \int_{G(B_j,B_k)} \lvert \nabla u \rvert^p d\mu.
\end{split}
\end{equation}

Summing over all the $B_j \in \mathcal{B_m}$, $S_j$ and $T_j$ we obtain
\[ 
\int_{(\Omega \setminus \Omega_m)\cup (\bigcup_{B\in \mathcal{B}_m} B)} \lvert \nabla u_m \rvert^p d\mu \ls \int_{D^{'}_m \cup E_m \cup F_m} \lvert \nabla u \rvert^p d\mu \leq \varepsilon,
\]
and thus we are done.

\end{proof}


\begin{thebibliography}{10}

\bibitem{AILP2024}
L. Ambrosio, T. Ikonen, D. Lu\v{c}i\'c, and E. Pasqualetto, \emph{Metric Sobolev Spaces I: Equivalence
of Definitions}, Milan J. Math. Vol. \textbf{92} (2024) 255--347.

\bibitem{APS2019}
L. Ambrosio, A. Pinamonti, and G. Speight,
\emph{Weighted Sobolev spaces on metric measure spaces},
J. Reine Angew. Math. \textbf{746} (2019), 39--65.

\bibitem{BB2003}
Z. M. Balogh, S. M. Buckley, \emph{Geometric characterizations of Gromov hyperbolicity}, Invent. Math. \textbf{153} (2003),
no. 2, 261--301.

\bibitem{BS2007}
J. Bj\"orn and N. Shanmugalingam, \emph{Poincar\'e inequalities, uniform
domains and extension properties for Newton-Sobolev functions in
metric spaces},
J. Math. Anal. Appl. \textbf{332} (2007), no. 1, 190--208.(2007), no. 1, 190--208.

\bibitem{BHK2001}
M. Bonk, J. Heinonen, P. Koskela, \emph{Uniformizing Gromov hyperbolic spaces}, Ast\'erisque No. \textbf{270} (2001), viii+99
pp.

\bibitem{CPSC1994}
V. Chiad\`o Piat and F. Serra Cassano, \emph{Some remarks about the density of smooth functions in weighted Sobolev spaces}
J. Convex Anal. \textbf{1} (1994), 135--142.

\bibitem{DCJ}
E. Durand-Cartagena, J.A. Jaramillo, \emph{Pointwise Lipschitz functions on metric spaces}, J. Math. Anal. Appl. \textbf{363} (2010) 525--548.

\bibitem{HKSmP}
P. Haj\l asz and P. Koskela, \emph{Sobolev met Poincar\'e}, Mem. Amer. Math. Soc., \textbf{145} (2000).

\bibitem{HKST}
J. Heinonen, P. Koskela, N. Shanmugalingam, J.T. Tyson: \emph{Sobolev spaces on metric measure spaces. An approach based on upper gradients}, vol. \textbf{27} of New Mathematical
Monographs, Cambridge University Press, Cambridge, (2015).

\bibitem{KRS2012}
A. K\"aenm\"aki, T. Rajala, and V. Suomala,
\emph{Existence of doubling measures via generalised nested cubes},
Proc. Amer. Math. Soc. \textbf{140} (2012), no. 9, 3275--3281.

\bibitem{K99}
P. Koskela, \emph{Removable sets for Sobolev spaces}, Ark. Mat. \textbf{37} (1999), no. 2, 291--304.

\bibitem{KRZ2017}
P. Koskela, T. Rajala, and Y. Zhang,
\emph{A density problem for Sobolev spaces on Gromov hyperbolic domains},
Nonlinear Anal. \textbf{154} (2017), 189--209.

\bibitem{KZ2016}
P. Koskela and Y. R.-Y. Zhang, \emph{A density problem for Sobolev spaces on planar domains}, Arch. Ration. Mech. Anal. \textbf{222} (2016), 1--14.

\bibitem{LP2024}
D. Lu\v{c}i\'c, Pasqualetto, E. \emph{Yet another proof of the density in energy of Lipschitz functions}, manuscripta math. 175, 421--438 (2024).

\bibitem{M2003}
M. Miranda, Jr., \emph{Functions of bounded variation on ``good'' metric
spaces}, J. Math. Pures Appl. \textbf{82} (2003), no. 8, 975--1004.

\bibitem{N2021}
D. Nandi, \emph{A density result for homogeneous Sobolev spaces}, Nonlinear Anal. \textbf{213} (2021), no. 112501.

\bibitem{NRS2019}
D. Nandi, T. Rajala, and T. Schultz
\emph{A density result for homogeneous Sobolev spaces on planar domains},
Potential Anal. \textbf{51} (2019), 483--498.

\bibitem{OR2021}
W. A. Ortiz and T. Rajala, \emph{A density result on Orlicz-Sobolev spaces in the plane}, J. Math. Anal. Appl. \textbf{503} (2021), no. 125329.

\bibitem{R2020}
T. Rajala, \emph{Approximation by uniform domains in doubling quasiconvex metric spaces}, Complex Anal. Synerg. \textbf{7} (2021), no. 4.

\bibitem{Z2013}
V.V. Zhikov, \emph{Density of smooth functions in weighted Sobolev spaces}, Dokl. Math. \textbf{88} (2013), 669--673.

\end{thebibliography}
\end{document}